\documentclass[11pt]{amsart}
\theoremstyle{plain}
\newtheorem{thm}{Theorem}[section]
\newtheorem{theorem}[thm]{Theorem}

\newtheorem{lemma}[thm]{Lemma}
\newtheorem{corollary}[thm]{Corollary}
\newtheorem{proposition}[thm]{Proposition}
\theoremstyle{definition}
\newtheorem{remark}[thm]{Remark}

\newtheorem{definition}[thm]{Definition}

\newtheorem{example}[thm]{Example}

\newtheorem{question}[thm]{Question}

\newtheorem{setup}[thm]{Set-up}
\numberwithin{equation}{section}

\newcommand{\sO}{{\mathcal O}}

\newcommand{\C}{{\mathbb C}}

\newcommand{\BP}{{\mathbb P}}
\newcommand{\Q}{{\mathbb Q}}
\newcommand{\R}{{\mathbb R}}

\newcommand{\Z}{{\mathbb Z}}

\title [Primitive automorphisms]{Pisot units, Salem numbers and higher dimensional projective manifolds with primitive automorphisms of positive entropy}

\author{Keiji Oguiso}

\address{Mathematical Sciences, the University of Tokyo, Meguro Komaba 3-8-1, Tokyo, Japan and Korea Institute for Advanced Study, Hoegiro 87, Seoul, 
133-722, Korea}
\email{oguiso@ms.u-tokyo.ac.jp}

\thanks{The author is supported by JSPS Grant-in-Aid (S) No 25220701, JSPS Grant-in-Aid (S) 15H05738, JSPS Grant-in-Aid (A) 16H02141, JSPS Grant-in-Aid (B) 15H03611, and by KIAS Scholar Program.}

\begin{document}

\maketitle

\begin{abstract} We show that, in any dimension greater than one, there are an abelian variety, a smooth rational variety and a Calabi-Yau manifold, with primitive birational automorphisms of first dynamical degree $>1$. We also show that there are smooth complex projective Calabi-Yau manifolds and smooth rational manifolds, of any even dimension, with primitive biregular automorphisms of positive topological entropy.
\end{abstract}

\section{Introduction}

Throughout this note, we work in the category of projective varieties defined over $\C$. 

The aim of this note is to give some affirmative answers (Theorems \ref{thm2}) to the following question \ref{ques1} (1) asked in \cite[Problem 1.1]{Og15} and a weaker version (2) in any dimension, for abelian varieties, smooth rational varieties and Calabi-Yau manifolds:

\begin{question}\label{ques1} 
\begin{enumerate} 
\item For each integer $\ell \ge 2$, is there a smooth projective variety of dimension $\ell$ with a primitive biregular automorphism of positive topological entropy? 
\item For each integer $\ell \ge 2$, is there a smooth projective variety of dimension $\ell$ with a primitive birational automorphism of first dynamical degree $> 1$?
\end{enumerate}
\end{question}

Question \ref{ques1} is related to birational geometry, 
complex dynamics, and also number theory and representation theory as we shall see (cf. Section 3). We recall the complex dynamical notion of topological entropy and closely related notions of dynamical degrees in Section 2, following \cite{Bo73}, \cite{Gr03}, \cite{Yo87}, \cite{DS05}, \cite{DN11} and \cite{Tr15}. Here we just recall that if $f \in {\rm Aut}\, (M)$, then $f$ is of positive entropy $h_{\rm top}(f) > 0$ if and only if the first dynamical degree $d_1(f) > 1$. The notion of primitivity of automorphism, introduced by De-Qi Zhang \cite{Zh09}, is purely algebro-geometric:

\begin{definition}\label{def1}
Let $M$ be a smooth projective variety of dimension $\ell \ge 2$ and $f \in {\rm Bir}\, (M)$. $f$ is called {\it imprimitive} if there are a dominant rational map $\pi : M \dasharrow B$ with connected fibers to a smooth projective variety $B$ with $0 < \dim\, B < \dim\, M$, and a rational map $f_B : B \dasharrow B$, necessarily $f_B \in {\rm Bir}\, (B)$, such that $\pi \circ f = f_B \circ \pi$. Here smoothness assumption is harmless by resolution of singularities \cite{Hi64}, as we work over $\C$. $f$ is {\it primitive} if it is not imprimitive.
\end{definition}

In the definition, $\pi$ and $f_B$ are not regular in general even if $f \in {\rm Aut}\, (M)$. If $M$ is a curve, i.e., if $\ell = 1$, then any automorphism of $M$ is primitive and of null topological entropy. So, from now, we assume that $\ell \ge 2$. A primitive birational automorphism is in some sense an {\it irreducible} automorphism in birational geometry. Compare with the following trivial (but related, see the proof of Theorem \ref{thm81}): 

\begin{remark}\label{rem1} Let $k$ be a field and $V$ a $k$-vector space of $\dim\, V = \ell \ge 1$. Let $f \in {\rm GL}\,(V)$. Then $f$ is not irreducible in the usual sense if and only if there are a $k$-vector space $W$ with $0 < \dim\, W < \dim\, V$, a surjective linear map $\pi : V \to W$ and $f_W \in {\rm GL}\, (W)$ such that $\pi \circ f = f_W \circ \pi$. The existence of irreducible $(V, f)$ with $\dim\, V = \ell$ depends on the field $k$. For instance $\ell = 1$ if $k = \C$, $\ell = 1$, $2$ if $k = \R$, while $\ell$ is arbitrary if $k = \Q$. 
\end{remark}

The existence of a primitive automorphism also depends on the classes of smooth projective varieties. Indeed, as again observed by De-Qi Zhang \cite{Zh09} (see also \cite[Theorem 2.2]{Og15}, \cite[Lemma 3.2, Theorem 3.3]{Og16} for the formulation here), we have the following rather strong constraint on projective varieties admitting primitive birational automorphisms: 

\begin{theorem}\label{thm1} Let $M$ be a smooth projective variety of dimension $\ell \ge 2$. We denote the Kodaira dimension of $M$ by $\kappa(M)$ and the irregularity of $M$ by $q(M)$. Then, if $M$ has a primitive birational automorphism $f \in {\rm Bir}\, (M)$, then ${\rm ord}\, f = \infty$ and $M$ falls into one of the following three exclusive classes:

(WRC) $\kappa(M) = -\infty$ and $q(M) = 0$;

(WCY) $\kappa(M) = 0$ and $q(M) = 0$;

(A) $M$ is birational to an abelian variety.

Assume in addition that the following two statements hold for $M$:

(1) If $\kappa(M) = 0$ and $h^1({\mathcal O}_M) = 0$, then $M$ is birational to a minimal Calabi-Yau variety $M'$, i.e., a normal projective variety $M'$ with only $\Q$-factorial terminal singularities such that ${\mathcal O}_{M'}(mK_{M'}) \simeq {\mathcal O}_{M'}$ for some $m >0$; and 

(2) If $\kappa(M) = -\infty$, then $M$ is uniruled. 

Then, $M$ falls into one of the following three exclusive classes:

(RC) $M$ is a rationally connected manifold, i.e., a smooth projective variety whose two general closed points are connected by a rational curve on $M$;

(CY) $M$ is birational to a minimal Calabi-Yau variety; or

(A) $M$ is birational to an abelian variety. 

\end{theorem}

\begin{remark}\label{rem2} The assumptions (1) and (2) hold if the minimal model problem and the abundance conjecture are affirmative in dimension $\ell$, so that they hold if $\ell \le 3$ by the minimal model theory and abundance theorem for projective threefolds due to Kawamata, Miyaoka, Mori and Reid (\cite{Mo88}, \cite{Ka92}, see also \cite{KMM87}, \cite{KM98}). 
\end{remark}

\begin{remark}\label{rem3} Obviously, smooth rational/unirational varieties are rationally connected. The most important classes of smooth projective varieties in (CY) or (WCY) are {\it Calabi-Yau manifolds} and projective {\it hyperk\"ahler manifolds} (see eg. \cite{GHJ03} for these manifolds). Recall that an $\ell$-dimensional simply-connected smooth projective variety $M$ is a Calabi-Yau manifold (resp. hyperk\"ahler manifold) if $H^0(\Omega_M^j) = 0$ for $0 < j < \ell$ and $H^0(\Omega_M^{\ell}) = \C\omega_M$ for a nowhere vanishing regular $\ell$-form $\omega_M$ (resp. $H^0(\Omega_M^{2}) = \C \eta_M$ for an everywhere non-degenerate regular $2$-form $\eta_M$, and therefore $\ell$ is even). 
\end{remark}

When $\ell = 2$, $3$, there are smooth projective rational varieties, Calabi-Yau manifolds and abelian varieties, with primitive bireguler automorphisms of positive entropy (See eg. \cite{Ca99}, \cite{BK09}, \cite{BK12}, \cite{Mc07}, \cite{Mc16}, \cite{Re12}, \cite{CO15}, \cite{Do16} for surfaces in several classes and \cite{OT14}, \cite{OT15} for threefolds).

Our main result is the following: 

\begin{theorem}\label{thm2} 
\begin{enumerate}
\item  For each $\ell \ge 2$, there is an $\ell$-dimensional abelian variety $A$ with a primitive biregular automorphism $f \in {\rm Aut}\, (A)$ of positive topological entropy. There is also an $\ell$-dimensional smooth projective variety $M$, birational to a minimal Calabi-Yau variety, with a primitive biregular automorphism of positive topological entropy. 
\item For each $\ell \ge 2$, there is a primitive birational automorphism $f \in {\rm Bir}\, (\BP^{\ell})$ of first dynamical degree $d_1(f) > 1$. 
\item For each $n \ge 2$, there is a $2n$-dimensional smooth rational variety $M$ with a primitive biregular automorphism $f \in {\rm Aut}\, (M)$ of positive topological entropy. 
\item For each $\ell \ge 2$, there is an $\ell$-dimensional smooth Calabi-Yau manifold $M$ with a primitive birational automorphism $f \in {\rm Bir}\, (M)$ of first dynamical degree $d_1(f) > 1$. 
\item For each $n \ge 2$, there is a $2n$-dimensional smooth Calabi-Yau manifold $M$ with a primitive biregular automorphism $f \in {\rm Aut}\, (M)$ of positive topological entropy.
\end{enumerate} 
\end{theorem}

We construct $A$ in (1) as the self product of an elliptic curve $E$ and its desired automorphism by using {\it Pisot units} (Definition \ref{def31}) in Theorem \ref{thm33}. $M$ in (1) is obtained by a standard resolution of the quotient variety $A/\langle -1_A \rangle$ (Corollary \ref{cor31}, see also Remark \ref{rem31}). Primitivity is checked by the product formula for the relative dynamical degrees (see Section 2 for summary). We also show the existence of rationally connected manifolds of dimension $4$ and $5$ with primitive biregular automorhisms of positive topological entropy as a byproduct of the existence of certain automorphisms (Corollary \ref{cor32} and the proof). We find a desried birational automorphism in (2) among monimial birational maps studied by \cite{FW12} and \cite{Li12}. Our proof of primitivity is similar to the one for (1) (Theorem \ref{thm63}). Our Calabi-Yau manifolds in (4) are Calabi-Yau manifolds of Wehler type studied in \cite{CO15} when $\ell$ is odd (Theorem \ref{thm81}). See also \cite{Og16-2} for a different construction. We construct desired manifolds and their biregular automorphisms in (3) and (5) using K3 surfaces $S$ with special automorphisms and their Hilbert schemes of $n$ points $S^{[n]}$ (Theorems \ref{thm61}, \ref{thm71}) and check the primitivity by reducing to the following remarkable result due to Bianco \cite[Main Theorem]{Bi16}:

\begin{theorem}\label{thm4} Let $M$ be a hyperk\"ahler manifold, $f \in {\rm Bir}\, (M)$ and $d_1(f)$ the first dynamical degree of $f$. If $d_1(F) >1$, then $f$ is primitive.
\end{theorem}

It might be worth mentioning that the following question remains open:

\begin{question}\label{qst1} 

(1) Are there Calabi-Yau manifolds, of {\it odd} dimension $\ge 5$, with primitive {\it biregular} automorphisms of positive topological entropy?

(2) Are there smooth rationally connected manifolds of {\it odd} dimension $\ge 7$, with primitive {\it biregular} automorphisms of positive topological entropy?
\end{question}
\par
\medskip

{\bf Acknowledgements.} I would like to express my thanks to Professors C. Bisi, S. Brandhorst, F. Catanese, I. Dolgachev, H. Esnault, Y. Kawamata, T. Kohno, T. Tsuboi, X. Yu, and especially Professors T. T. Truong and D.-Q. Zhang, for inspiring discussions and encouragement. I would like to express my thanks to Mr. S. Brandhorst for pointing out me an important paper \cite{Bi16}.  

\section{Entropy, Dynamical degrees and relative dynamical degrees}

In this section we briefly recall definitions of topological entropy, dynamical degrees, relative dynamical degrees and their basic properties, used in this paper. No new result is included. Main references here are Dinh-Sibony \cite{DS05}, Dinh-Nguyen \cite{DN11}, Dinh-Nguyen-Troung \cite{DNT12} and Troung \cite{Tr15}, \cite{Tr16}. We state the results for birational automorphisms of smooth projective varieties, while all the results below are valid for any compact K\"ahler manifolds and dominant meromorphic selfmaps (if we use K\"ahler forms instead of Fubini-Study forms or hyperplane classes). 

\subsection{Topological entropy and Gromv-Yomdin theorem}

Let $M = (M, d)$ be a compact metric space and $f : M \rightarrow M$ a continuous surjective selfmap of $M$. The {\it topological entropy} of $f$ is the fundamental invariant that measures {\it how fast two general points spread out under the action of the semi-group $\{f^n \vert n \in {\mathbf Z}_{\ge 0}\}$}. For the definition, we define the new distance $d_{f, n}$ on $X$, depending on $f$ and $n$, by 
$$d_{f, n}(x, y) = {\rm max}_{0 \le j \le n-1} d(f^j(x), f^j(y))\,\, {\rm for}\,\, x, y \in X.$$
Let $\epsilon > 0$ be a positive real number. 
We call two points $x, y \in M$ $(n, \epsilon)$-{\it separated} if $d_{f, n}(y, x) \ge \epsilon$, and a subset $F \subset M$ $(n, \epsilon)$-{\it separated} if any two distinct points of $F$ are $(n, \epsilon)$-separated. 
Let 
$$N_{d}(f, n, \epsilon) := {\rm Max}\, \{\vert F \vert\,\, \vert\,\, F \subset M\,\, {\rm is}\,\,  (n, \epsilon)-{\rm separated}\,\, \}\,\, .$$
Note that $N_{d}(f, n, \epsilon)$ is a well-defined positive integer, as $M$ is compact. The following definition was introduced by Bowen \cite{Bo73}:

\begin{definition}\label{def21} The {\it topological entropy} of $f$ is defined by:
$$h_{{\rm top}}(f) := {\rm lim}_{\epsilon \to +0} {\rm limsup}_{n \to \infty} \frac{\log N_d(f, n, \epsilon)}{n}\,\, .$$
\end{definition}
Note that $h_{{\rm top}}(f)$ does not depend on the metric $d$ of the topological space $M$.  

Let $M$ be a smooth projective variety and $f: M \to M$ a surjective morphism. Then $f$ is continuous in the classical metric topology given by the Fubini-Study metric $d(*, **)$ under any embedding $M \subset \BP^N$. So, one can speak of the topological entropy of $f$. One of the most fundamental properties of the topological entropy is the following cohomological characterization due to Gromov and Yomdin (\cite{Gr03}, \cite{Yo87}):
\begin{theorem}\label{thm21}
Let $M$ be a smooth projective variety of dimension $\ell$ and $f : M \to M$ a surjective morphism. Then, $d_p(f) = r_{p}(f)$ and 
$$h_{{\rm top}}(f) = \log {\rm max}_{0 \le p \le \ell}\, d_{p}(f) = \log {\rm max}_{0 \le p \le \ell}\, r_{p}\,(f)  = \log r\,(f)\,\, .$$
Moreover, $h_{\rm top}(f) > 0$ if and only if $d_1(f) > 1$.
\end{theorem}
Here $r_p\,(f)$ and $r\,(f)$ are the spectral radii of $f^*\vert H^{p,p}(X)$ and $f^* \vert \oplus_{p=0}^{2\ell} H^{p}(M, {\mathbf Z})$ respectively. The {\it $p$-th dynamical degree} $\lambda_p(f)$ is defined (and well-defined) by:
$$d_p(f)=\lim _{n\rightarrow\infty}(\int_{X} (f^n)^{*}(\omega^{p}) \wedge \omega^{k-p})^{\frac{1}{n}} = \lim _{n\rightarrow\infty}((f^n)^{*}(H^p).H^{\ell -p})_{M}^{\frac{1}{n}} \ge 1\,\, .$$
Here $H$ is the hyperplane class of $M$ (under any embedding $M \subset \BP^N$) and $\omega$ is the Fubini-Study form, i.e., the positive closed $(1,1)$-form induced by the Fubini-Study form of $\BP^N$. 

\begin{remark}\label{rem20} 
\begin{enumerate}
\item The equality $d_p(f) = r_p(f)$ shows that one can replace $H^{p, p}(M)$ by $N^p(M)$, provided that $M$ is projective (see eg. \cite{Tr15}, \cite{Tr16}). Here $N^p(M)$ is the finitely generated free $\Z$-module consisting of the numerical equivalence classes of algebraic $p$-cocycles. 

\item The first equality $d_p(f) = r_{p}(f)$ is purely linear algebraic and indeed just follows from $(f^n)^* = (f^*)^n$ for all $n \ge 0$, which is trivial for automorphisms. So, for instance, if $f \in {\rm Bir}\, (M)$ is isomorphic in codimension one, then $d_1(f) = r_1(f)$, as $(f^n)^* = (f^*)^n$ on $N^1(M)$. 
\end{enumerate}
\end{remark}

\subsection{Dynamical degrees} 
One can define $r_p(f)$, $r(f)$ and $d_p(f)$ also for birational self map $f : M \dasharrow M$, i.e., for $f \in {\rm Bir}\, (M)$. Precisely, for each $n \ge 1$, we define the pull back $(f^n)^*$ by $(f^n)^*(a) := (p_{1,n})_* ((p_{2, n})^*(a))$ and $d_p(f)$ by the same formula above. Here $p_{1, n} : \tilde{M}_n \to M$ and $p_{2, n} : \tilde{M}_n \to M$ be any birational {\it morphisms} from a smooth projective variety $\tilde{M}_n$, depending on $f$ and $n$, such that $f^n = p_{2, n} \circ (p_{1, n})^{-1}$. In general $r_p(f)$ and $r(f)$ (we may just take $n=1$) are not birational invariant, as the standard Cremona involution of $\BP^2$, which is not isomorphic in codimension one, shows. On the other hand, the dynamical degrees $d_p(f)$ are well-defined birational invariant, as Dinh-Sibony (\cite{DS05}) proved (See also \cite{Tr15} for purely algebro-geometric proof):
\begin{theorem}\label{thm22}
Let $M$ and $M'$ be smooth projective varieties and $\mu : M \dasharrow M'$ a birational map. Let $f \in {\rm Bir}\, (M)$ and let $f' = \mu \circ f \circ \mu^{-1}$ be the birational automorphism of $M'$ induced by $f$ and $\mu$. Then, $d_p(f)$ is well defined and $d_p(f) = d_p(f')$ for all $p$. 
\end{theorem}

\subsection{Relative dynamical degrees}

\begin{setup}\label{set21}
Let $M$ and $B$ be smooth projective varieties of dimension $\ell$ and $b$ respectively, $f \in {\rm Bir}\, (M)$ and $\pi : M \to B$ be an $f$-equivariant surjective morphism so that there is $f_B \in {\rm Bir}\, (B)$ with $\pi \circ f = f_B \circ \pi$.
Here we do not require that $\pi$ has connected fibers. 
Let $H_M$ and $H_B$ be the hyperplane class of $M$ and $B$ (under any embedding to projective spaces) and $\omega_M$ and $\omega_B$ the Fubini-Study forms induced by the embeddings.
\end{setup}

\begin{remark}\label{rem21}If $f \in {\rm Bir}\, (M)$ is imprimitive and $\pi : M \dasharrow B$ be $f$-equivariant, then by replacing $M$ by $\tilde{M}$, a resolution $\mu : \tilde{M} \to M$ of indeterminacy of $\pi$, $f$ by $\tilde{f} = \mu^{-1} \circ f \circ \mu$ and $\pi$ by the induced morphism $\pi \circ \mu$, we always obtain $\tilde{f}$-equivariant surjective morphism $\pi \circ \mu : \tilde{M} \to B$ as in Set-up \ref{set21}. In this modification, we may loose biregularity of $\tilde{f}$ in general even if $f \in {\rm Aut}\, (M)$ but the values $d_p(f)$ remain the same, i.e., $d_p(\tilde{f}) = d_{p}(f)$ by Theorem \ref{thm22}. So, when we compute the dynamical degree $d_p(f)$, we may assume $\pi$ is a surjective {\it morphism}.
\end{remark}

The notion of the relative dynamical degrees $d_p(f|\pi )$ is defined by Dinh-Nguyen \cite{DN11} as follows: 
\begin{definition}\label{def22} Under Set-up \ref{set21}, the {\it relative $p$-th dynamical degree} $d_p(f|\pi )$ is defined as
$$\lim _{n \to\infty}(\int _M(f^n)^*(\omega _M^p) \wedge \omega _M^{\ell -b -p} \wedge \pi ^*(\omega _B^{b}))^{\frac{1}{n}} = \lim _{n\to\infty}((f^n)^*(H_M^p).H_M^{\ell -b -p}.\pi^*(H_B^{b}))_M^{\frac{1}{n}} \ge 1\,\, .$$
\end{definition} 
$d_p(f|\pi )$ is well-defined by \cite{DN11}. We may regard $\pi^*(H_B^{b})$ as a general fiber class of $\pi$, up to positive constant multiple. So, the relative dynamical degree could be considered as the dynamical degree of $f$ {\it restricted to a general fiber} $F$.

The following important result, called the {\it product formula}, was proved by Dinh-Nguyen \cite{DN11} (See also \cite{Tr15} for purely algebro-geometric approach):

\begin{theorem}\label{thm23} Under Set-up \ref{set21}, 
$d_p(f)= \max_{\max\{0,p-\ell +b\}\leq j\leq 
\min\{p,b\}}d_j(g)d_{p-j}(f|\pi)$. 
\end{theorem}
The following conclusions (1) and (2) of Theorem \ref{thm23} are observed by \cite[Corollary 1.2]{DN11} and \cite[Theorem 4.1 (1)]{OT15} respectively, and will be used in the proof of our main result:
\begin{corollary}\label{cor21} Let $M$ and $B$ be smooth projective varieties, $f \in {\rm Bir}\, (M)$ be a birational automorphism of $M$, $\pi : M \dasharrow B$ an $f$-equivariant dominant rational map and $f_B$ be the induced birational automorphism of $B$. Then:
\begin{enumerate}

\item If $\pi$ is generically finite, i.e., if $\dim\, M = \dim\, B$, then $d_p(f) = d_p(f_B)$ for every $p$. 

\item If $0 < \dim\, B < \dim\, M$, then $d_1(f) \le d_2(f)$. In particular, $f$ is primitive if $d_1(f) > d_2(f)$. 

\end{enumerate}
\end{corollary} 

\section{Pisots units and Salem numbers.}

In this section, we briefly recall definitions and properties of Pisot units and Salem numbers, which we will use. Theorem \ref{prop41} is a slight generalization of \cite[Theorem 4.1]{EOY16} which is used in the proof of Theorem \ref{thm81} and will be applicable in other situation. 

First, we recall the definition of Pisot unit and properties we will use.

\begin{definition}\label{def31}
Let $\overline{\Z} \subset \C$ be the ring of algebraic integers.  
A real algebraic 
integer $\alpha \in \overline{\Z} \cap \R$ is called a {\it Pisot number} if $|\alpha| > 1$ and $|\alpha'| < 1$ for all Galois conjugates $\alpha' \not= \alpha$ of $\alpha$ over $\Q$. Here and hereafter, for a complex number $a$, we denote the complex conjugate of $a$ by $\overline{a}$ and define $|a| := \sqrt{a \overline{a}}$. The degree of the Pisot number $\alpha$ is the degree of the minimal polynomial of $\alpha$ over $\Z$ or equivalently over $\Q$. A Pisot number $\alpha$ of degree $d$ is called a {\it Pisot unit} if $\alpha$ is invertible in $\overline{\Z}$, i.e. $\prod_{j=1}^{d}\alpha_j = \pm 1$, where $\alpha_i$ ($1 \le j \le d$) are the Galois conjugates of $\alpha$.  We note that if $\alpha$ is a Pisot number or a Pisot unit of degree $d$ then so is $-\alpha$. So, {\it from now on, we may and will assume that Pisot numbers are greater than $1$}.   
\end{definition}

\begin{example}\label{ex31}

(1) Let $a$ be an integer such that $a \ge 3$. Then, by definition, the largest root of $X^2 - aX + 1 = 0$ is a Pisot unit of degree $2$. Similarly, for a positive integer $b$, the largest root of $X^2 - bX - 1 = 0$ is a Pisot unit of degree $2$.  

(2) The largest real roots $\alpha_3 = 1.324\ldots\,\, ,\,\, \alpha_4 = 1.380\ldots\,\, ,\,\, \alpha_5 = 1.443\ldots$
 of the equations $X^3 -X -1 = 0\,\, ,\,\, X^4 -X^3 -1 = 0\,\, ,\,\, X^5 -X^4 -X^3 + X^2 -1 = 0$ are Pisot units of degree $3$, $4$, $5$ respectively. These three Pisot units $\alpha_3 < \alpha_4 < \alpha_5$ are the smallest three positive Pisot numbers (\cite[Theorem 7.2.1]{BDGPS92}). 

\end{example}

There are plenty of Pisot units (\cite[Theorem 5.2.2]{BDGPS92}):
\begin{theorem}\label{thm31}
Let $d \ge 2$ be any positive integer and $K$ any real field extension of $\Q$ of degree $d = [K : \Q]$ (for instance $K = \Q(\sqrt[d]{2})$). Then there is a Pisot unit $\alpha \in K$ of degree $d = [K:\Q]$. 
\end{theorem}

The first dynamical degrees of birational automorphisms of rational surfaces, which are not conjugate to biregular automorphisms, are necessarily Pisot numbers (\cite{DF01}, \cite{BC13}). They are not necessarily Pisot units 
(cf. Section 4).

Next we recall the definition of Salem number (including quadratic units) and properties we will use. 

\begin{definition}\label{def41} A polynomial $P(X) \in \Z [X]$ is called a {\it Salem polynomial} if it is irreducible over $\Z$, monic,  of even degree $2d \ge 2$ and the complex zeroes of $P(x)$ are of the form ($1 \le i \le d-1$):
$$a > 1\,\, ,\,\, 0 < \frac{1}{a} < 1\,\, ,\,\, \alpha_i\,\, ,\,\, \alpha_{i+d-1} := \overline{\alpha}_i \in S^1 := 
\{z \in \C\, \vert\, \vert z \vert = 1\} \setminus \{\pm 1\}\,\, .$$ 
A {\it Salem number} is the largest real root $a >1$ of a Salem polynomial $P(X)$ and we call the degree of $P(X)$ the degree of $a$. By definition, Salem numbers are always units, i.e., in $\overline{\Z}^{\times}$, and of even degree. 
\end{definition} 

\begin{example}\label{ex41} Unlike Pisot numbers, it is unknown which is the smallest Salem number. The smallest known Salem number is the {\it Lehmer number} 
$\lambda_{\rm Lehmer} = 1.17628\ldots$, the real root $>1$ of the following Salem polynomial of degree $10$:
$$X^{10} + X^9 -X^7 -X^6 -X^5 -X^4 -X^3 + X + 1\,\, .$$ 
It is conjectured that the Lehmer number is the smallest Salem number and in fact it is the smallest one in degree $\le 40$ (see eg. the webpage \cite{Mo03}). Recall that the first dynamical degree $d_1(f)$ of a smooth surface automorphism $f$ is Salem number if 
$d_1(f) > 1$ (cf. Proposition \ref{prop41}). McMullen \cite[Theorem A.1]{Mc07} shows that the Lehmer number is the smallest Salem number among the first dynamical degrees $d_1(f) \not= 1$ of smooth surface automorohisms $f$ (\cite[Theorem A.1]{Mc07}). He also shows that the Lehmer number is realized as $d_1(f)$ of a rational surface automorphism $f$ and a projective K3 surface automorphism $f$ (\cite{Mc07}, \cite{Mc16}).
\end{example}

We have the following purely lattice theoretic result (for a smooth projective surface $S$, we can apply Theorem \ref{prop41} (1) for $L = N^1(S)$): 
 
\begin{theorem}\label{prop41} Let $L = (\Z^n, (*, **))$ be a 
hyperbolic lattice of rank $n$, i.e. a free $\Z$-module of rank $n$ with an integral non-degenerate bilinear form $(*, **)$ of signature $(1, n-1)$. Let $L_{\R} = L \otimes \R$ and let ${\rm O}(L)$ (resp. ${\rm O}(L_{\R})$) be the orthogonal groups of the hyperbolic lattice $L$ (resp. of the real hyperbolic space $L_{\R}$) with respect to $(*, **)$. We also define 
$${\rm O}^{+}(L_{\R}) := \{f \in {\rm O}(L_{\R})\, |\, f(P) = P\}\,\, .$$
Here $P$ is one of the two connected components of $\{x \in L_{\R}\, |\, (x,x) > 0\}$. Then:
\begin{enumerate}
\item Let $f \in {\rm O}(L) \cap {\rm O}^{+}(L_{\R})$. Then the characteristic polynomial $\Phi_f(x) \in \Z[x]$ of $f$ is the product of cyclotomic polynomials (possibly empty) and at most one Salem polynomial (possibly empty), counted with multiplicities.  
\item Assume that the rank $n$ of $L$ is even. Let $G$ be a subgroup of ${\rm O}(L)$. Assume that the action of $G$ on $L_{\R}$ is irreducible, i.e., there is no $G$-stable linear subspace of $L_{\R}$ other than $\{0\}$ and $L_{\R}$. Then, there is $f \in G$ whose characteristic polynomial $\Phi_f(x)$ is a Salem polynomail of degree $n$. 
\end{enumerate}
\end{theorem}    
\begin{proof} The assertion (1) is well-known, see eg. \cite[Proposition 2.5]{Og06} for a self-contained proof. 

We shall prove the assertion (2). Note that ${\rm O}(L_{\R})$ is a real algebraic group and the special real orthogonal group ${\rm SO}(L_{\R})$ is an algebraic subgroup of ${\rm O}(L_{\R})$ of index two. Thus by \cite[Proposition 1]{BH04}, the Zariski closure of $G$ in ${\rm O}(L_{\R})$ 
is either ${\rm SO}(L_{\R})$ or ${\rm O}(L_{\R})$. We also note that 
$g^2 \in {\rm SO}^{+}(L_{\R}) := {\rm O}^{+}(L_{\R}) \cap {\rm SO}(L_{\R})$ for $g \in O(L_{\R})$. Let $P_n \simeq \R^n$ be the real affine space consisting of the real monic polynomial of degree $n$. Then, as $n$ is even, the same proof in \cite[Theorem 4.1]{EOY16} applied for the real algebraic morphism 
$$G \to P_n\,\, ;\,\, g \mapsto \Phi_{g^2}(x)$$
concludes that there is $f \in G$ such that the characteristic polynomail $\Phi_{f^2}(x)$ of $f^2$ is a Salem polynomial of degree $n$. 
\end{proof}

\section{Abelian varieties with primitive automorphisms of positive topological entropy and a few applications.}

Our main results of this section are Theorems \ref{thm31}, \ref{thm32} and 
Corollaries \ref{cor31}, \ref{cor32}, from which Theorem \ref{thm2} (1) 
follows. 

Pisot numbers, or more precisely Pisot units, play important roles in primitive automorphisms of abelian varieties:

\begin{theorem}\label{thm32}
Let $d$ be an integer such that $d \ge 2$. Let $A$ be a $d$-dimensional abelian variety and $f \in {\rm Aut}\, (A)$. Assume that $f^*|H^0(A, \Omega_A^1)$ has a Pisot number $\alpha >1$ of degree $d$ as its eigenvalue. Then $\alpha$ is a Pisot unit and $f$ is a primitive automorphism of $A$ of positive topological entropy $h_{\rm top}(f) = 2 \log \alpha > 0$.  
\end{theorem}

\begin{proof} By the Hodge decomposition theorem, we have
$$H^1(A, \Z) \otimes_{\Z} \C = H^1(A, \C) = H^0(A, \Omega_A^1) \oplus \overline{H^0(A, \Omega_A^1)}\,\, .$$
This decomposition is compatible with the action of $f^*$. Thus 
the eigenvalues of $f^*|H^1(A, \C)$ are $\alpha_j\,\, ,\,\, \overline{\alpha}_j \,\, (1 \le j \le d)$,  
counted with multiplicities. Here $\alpha_1 := \alpha$ and $\alpha_j$ ($1 \le j \le d$) are the Galois conjugates of $\alpha$. As $f$ is an automorphism, $f^*|H^1(A, \Z)$ is invertible over $\Z$. Thus the determinant of $f^{*}|H^1(A, \Z)$, which is 
$\pm \prod_j |\alpha_j|^2$, is $\pm 1$. Hence $\prod_j \alpha_j$ is $\pm 1$ as $\prod_j \alpha_j \in \Z$. This shows that $\alpha$ is a Pisot unit. As $\alpha = \alpha_1$ is a positive Pisot number, we may order $\alpha_j$ so that 
$$|\alpha_1| = \alpha > 1 > |\alpha_2| \ge \ldots \ge |\alpha_d|\,\, .$$
As $A$ is an abelian variety, we have
$$H^{p, p}(A) = \wedge^{p} H^0(A, \Omega_A^1) \otimes_{\C} \wedge^{p} \overline{H^0(A, \Omega_A^1)}\,\, .$$
It follows that 
$$d_k(f) = \prod_{j=1}^{k} |\alpha_j|^2\,\, .$$
In particular, 
$$d_1(f) = |\alpha_1|^2\,\, ,\,\, d_2(f) = |\alpha_1|^2|\alpha_2|^2\,\, ,\,\, d_3(f) = |\alpha_1|^2|\alpha_2|^2|\alpha_3|^2\,\, ,\,\, \ldots$$
and therefore, $h_{\rm top}(f) = 2 \log \alpha > 0$ by $\alpha = |\alpha_1| > 1$ and $|\alpha_j| < 1$ for $j \ge 2$ and by Theorem \ref{thm21}. 

As $|\alpha_2| < 1$, it follows that $d_1(f) > d_2(f)$.   
Hence $f$ is primitive by Corollary \ref{cor21} (2). This completes the proof of Theorem \ref{thm32}.
\end{proof}

The first part of Theorem \ref{thm2} (1) follows from the next theorem:

\begin{theorem}\label{thm33}
Let $d$ be an integer such that $d \ge 2$. Let $\alpha >1$ be any Pisot unit of degree $d$, whose existence is guaranteed by Theorem \ref{thm31}. Let $E$ be any elliptic curve. Then the $d$-dimensional abelian variety $E^d$ admits a primitive automorphism $f \in {\rm Aut}_{{\rm group}}\, (E^d)$ with
$$h_{\rm top}(f) = 2 \log \alpha > 0\,\, .$$
\end{theorem}
\begin{proof}
Let $S_d(X) = X^d + a_dX^{d-1} + \ldots + a_2X + a_1 \in \Z[X]$
be the minimal polynomial of $\alpha$. Note that $a_1 = \pm 1$ as $\alpha$ is a Pisot unit. Consider the matrix $M_d = (m_{ij}) \in M_d(\Z)$ whose entries $m_{ij}$ are $0$ except 
$$m_{dj} = -a_j\,\, ,\,\, m_{i, i+1} = 1\,\, (1 \le j \le d\,\, ,\,\, 1 \le i \le d-1)\,\, ,$$
associated to the polynomial $S_d(X)$. 
For instance
$$M_2 = \left(\begin{array}{rr}
0 & 1\\
-a_1 & -a_2\\
\end{array} \right)\,\, ,\,\, M_3 = \left(\begin{array}{rrr}
0 & 1 & 0\\
0 & 0  & 1\\
-a_1 & -a_2 & -a_3\\
\end{array} \right)\,\, , \,\, M_4 = \left(\begin{array}{rrrr}
0 & 1 & 0 & 0\\
0 & 0  & 1 & 0\\
0 & 0 & 0 & 1\\
-a_1 & -a_2 & -a_3 & -a_4\\
\end{array} \right)\,\, .$$
By definition of $M_d$, the characteristic polynomial of $M_d$ is $S_d(X)$ and, as $a_1 = \pm 1$, we have $M_d \in {\rm GL}_d(\Z)$. 
Let $(z_1, z_2, \ldots, z_d)$ be the standard complex coordinates of the universal cover $\C^d$ of $E^d$. Then, as $M_d \in {\rm GL}_d(\Z)$, we can uniquely define the group automorphism $f \in {\rm Aut}_{\rm group}(E^d)$ by 
$$f^*(z_1, z_2, \ldots, z_d)^t = M_d(z_1, z_2, \ldots, z_d)^{t}\,\, .$$ 
Here $^t$ is the transpose. Then 
$f^*| H^0(E^d, \Omega_{E^d}^1) = M_d$ under the basis $\langle dz_i \rangle_{i=1}^{d}$ of $H^0(E^d, \Omega_{E^d}^1)$. The characteristic polynomial of $f^*|H^0(E^d, \Omega_{E^d}^1)$ is then $S_d(X)$, the minimal polynomial of Pisot unit $\alpha$ of degree $d$. The result now follows from Theorem \ref{thm32}. 
\end{proof}

The second part of Theorem \ref{thm2} (1) follows from the following: 

\begin{corollary}\label{cor31}
Let $d$ be an integer such that $d \ge 3$. Let $\alpha >1$ be any Pisot unit of degree $d$, whose existence is guaranteed by Theorem \ref{thm31}. Then there is a $d$-dimensional smooth projective variety $M$, birational to a minimal Calabi-Yau variety, such that $M$ admits a primitive automorphism $f$ with $h_{\rm top}(f) = 2 \log \alpha > 0$. 
\end{corollary}
\begin{proof}
Let $M$ be the blow up at the maximal ideals of the singular points of the quotient variety $\overline{M} := E^d/\langle -1_{E^d}\rangle$. Then $\overline{M}$ is a minimal Calabi-Yau variety as $d \ge 3$, and $M$ is a smooth projective variety birational to $\overline{M}$. The automorphism $f \in {\rm Aut}_{{\rm group}}\, (E^d)$ in Theorem \ref{thm32} descends to an automorphism $f_M \in {\rm Aut}\, (M)$ of $M$ as $f\circ (-1_{E^d}) = (-1_{E^d}) \circ f$ and by the universality of the blow up. By Corollary \ref{cor21} (1), $d_p(f_M) = d_{p}(f)$ for all $p$. Thus, by Theorem \ref{thm21}, $h_{\rm top}(f_M) = h_{{\rm top}}(f)$, which is $2\log \alpha$ by Theorem \ref{thm33}. $f_M$ is also primitive by Lemma \ref{lem32} below.
\end{proof}

\begin{lemma}\label{lem32}
Let $U$ and $V$ be smooth projective varieties of the same dimension $d$ 
and $\mu : U \dasharrow V$ a dominant rational map (necessarily generically finite). Let $f_U \in {\rm Bir}\,(U)$ and $f_V \in {\rm Bir}\,(V)$. Then, $f_V$ is primitive if $f_U$ is primitive and $\mu \circ f_U = f_V \circ \mu$. 
\end{lemma} 

\begin{proof}
If $\pi : V \dasharrow B$ is an $f_{V}$-equivariant dominant rational map of connected fibers, then the Stein factorization of $\pi \circ \mu : U \dasharrow B$ is an $f_U$-equivariant dominant rational map of connected fibers, with $(\dim\, B$)-dimensional base space. As $f_U$ is primitive and $\dim\, V = \dim\, U = d$, it follows $\dim\, B = 0$ or $\dim\, B = d$. Hence $f_V$ is primitive. 
\end{proof}

\begin{remark}\label{rem31}
The manifolds $M$ constructed in the proof of Corollary \ref{cor31} are not birational to a Calabi-Yau manifold nor a hyperk\"ahler manifold when $\dim M \ge 3$. The reason is as follows. Notice that $\overline{M}$ is birational to $M$, has only isolated terminal singularities with numerically trivial canonical divisor and $\pi_1(\overline{M} \setminus S)$ is an infinite group for any proper closed algebraic subset $S$ such that ${\rm Sing}\,\overline{M} \subset S$. On the other hand, if $\overline{M}$ would be birational to a Calabi-Yau manifold or a hyperk\"ahler manifold, then they would be isomorphic in codimension one (see eg. \cite{Ka08}) so that $\pi_1(\overline{M} \setminus S) = \{1\}$ for some $S$ above of codimension $\ge 2$, a contradiction. 
\end{remark}

\begin{corollary}\label{cor32}
Let $\alpha_d$ ($d =3$, $4$, $5$) be the first three smallest positive Pisot units of degree $d$ respectively (Example \ref{ex31}). 
Then the logarithm $2 \log \alpha_d$ ($d = 3$, $4$, $5$) is realized as the topological entropy of a primitive automorphism of an abelian varieties of dimension $d=3$, $4$, $5$ respectively. Moreover, $2 \log \alpha_3$ is also realized as the topological entropy of a primitive automorphism of a $3$-dimensional Calabi-Yau manifold and a $3$-dimensional smooth rational variety, and $2 \log \alpha_d$ ($d = 4$, $5$) is also realized as the topological entropy of a primitive automorphism of a $d$-dimensional smooth rationally connected variety $R_d$. 
\end{corollary}
\begin{proof}
In Theorem \ref{thm32}, we choose $E$ to be $E_{\omega} = \C/\Z + \omega\Z$ ($\omega = (-1 + \sqrt{-3})/2$, the primitive third root of $1$). Then $f_d := f \in {\rm Aut}_{{\rm group}}\, (E_{\omega}^d)$ in Theorem \ref{thm32}, associated to $\alpha_d$ ($d = 3$, $4$, $5$), satisfies the first requirement. 

Let $\mu : R_d \to \overline{R}_d$ be the blow up at the maximal ideals of the singular points of the quotient variety $\overline{R}_d := E_{\omega}^d/\langle -\omega I_d \rangle$. Then $R_d$ is a smooth projective variety. Moreover, the automorphism $f_d \in {\rm Aut}_{{\rm group}}\, (E^d)$ descends to an automorphism $f_{R_d} \in {\rm Aut}\, (R_d)$ of $R_d$ and $f_{R_d}$ is primitive of positive topological entropy $2\log \alpha_d$, exactly for the same reason as in Corollary \ref{cor31}. 

Let $V_3$ be the blow up at the maximal ideals of the singular points of the quotient variety $E_{\omega}^3/\langle \omega I_3 \rangle$. Then $V_3$ is a smooth Calabi-Yau threefold and the automorphism $f_{V_3} \in {\rm Aut}\, (V_3)$ of $V_3$ induced by $f_3$ satisfies all the required properties. 

$R_3$ is rational by \cite{OT15}, $R_4$ is unirational by \cite{COV15} and $R_5$ is rationally connected by \cite[Corollary25]{KL09}. 

{\it Here, we shall give an alternative uniform proof of rational connectedness of $R_d$ ($d=3$, $4$, $5$), using the fact that $f_{R_d} \in {\rm Aut}\,(R_d)$ is primitive.} 

By construction, $\overline{R}_d$ ($d=3$, $4$, $5$) has numerically trivial canonical divisor and only isolated singular points which are the image of the fixed points of $\langle -\omega I_3 \rangle$. Let $P \in \overline{R}_d$ be the image of the origin of $E_{\omega}^d$. As $d \le 5$, $\overline{R}_d$ is klt but not canonical at $P$. Let $E \subset R_d$ be the exceptional divisor lying over $P$. Then we have $K_{R_d} \equiv -aE + E'$ with $a >0$. Here $E'$ is a divisor whose support lies over the singular points of $\overline{R}_d$ other than $P$.

Let $y \in R_{d}$ be a general point of $R_d$ and set $x = \mu(y)$. As $y$ is general, $x$ is a smooth point of $\overline{R}_d$. Choose an ample divisor $H$ of $\overline{R}_d$. Then, there is a positive integer $m$ (may depends on $y$) and a complete intersection curves 
$\overline{C} = H_1 \cap H_2 \cap \ldots \cap H_{d-1}$ ($H_i \in |mH|$) such that $\overline{C}$ is irreducible, $\overline{C} \ni x$, $\overline{C} \ni P$ and $\overline{C}$ contains no other singular points of $\overline{R}_d$. This is possible, as $\overline{R}_d$ has only finitely many singular points. Let $C \subset R_d$ be the strict transform of $\overline{C}$. Then $(K_{R_d}.C) = -a(E.C) + (E'.C) = -a(E.C) <0$. 
Note that $y \in C$ as $x \in \overline{C}$ and $x$ is a smooth point. Therefore, for a general point $y \in R_{d}$, there is an irreducible curve $C \subset R_d$ such that $y \in C$ and $(K_{R_d}.C) < 0$.  Hence $R_d$ is uniruled by the numerical criterion of the uniruledness due to Miyaoka-Mori (\cite{MM86}, see also \cite[Chap. IV, Theorem 1. 13]{Ko96}). Here one can also apply \cite[Lemma 2.4]{HMZ14} to conclude the uniruledness of $R_d$ ($3 \le d \le 5$). 

{\it As one of the referees pointed out, the argument here shows that the quotient variety of an abelian variety by a finite group is uniruled if it admits a singular point worse than canonical singularity}. (See also \cite{KL09}, \cite{GHS03} for more details.) 

Now consider the maximal rationally connected fibration $\pi : R_d \dasharrow B$ of $R_d$. It is unique up to bitational equivalence (\cite{KMM92}, \cite[Chap. IV, Theorem 5.5]{Ko96}). In particular, $\pi$ is $f_{R_d}$-equivariant. Here $B$ is not uniruled by \cite{GHS03}. Thus $\dim\, B < \dim R_d$ as $R_d$ is uniruled. As $f_{R_d}$ is primitive, it follows that $\dim\, B = 0$. Hence $R_d$ is rationally connected. 
\end{proof}

The following question is yet open and is of its own interest (see \cite{COV15} for some attempt):

\begin{question}\label{quet31} Are $R_4$ and $R_5$ constructed in our proof of Corollary \ref{cor32} rational?  
\end{question}
Note that $R_1$ and $R_2$ are obviously rational and $R_3$ is also rational (\cite{OT15}), while  $R_d$ is never rational (even never uniruled) for $d \ge 6$, as $E_{\omega}^d/\langle -\omega I_d \rangle$ ($d \ge 6$) has only canonical singularities with numerically trivial canonical divisor. 

\section{Hilbert schemes of points and hyperk\"ahler manifolds with primitive automorphisms of positive topological entropy.}

Our main result of this section is Theorem \ref{thm41}. Salem numbers also naturally appear as the dynamical degrees of automorphisms of projective hyperk\"ahler manifolds (\cite{Og09}):
\begin{proposition}\label{prop42} Let $M$ be a (not necessarily projective) hyperk\"ahler manifold of dimension $2m$ and $f \in {\rm Aut}\, (M)$. Then, the $p$-th dynamical degree $d_p(f)$ is either $1$ or a Salem number of degree $\le b_2(M)$. Moreover, $d_p(f) = d_{2m-p}(f) = d_1(f)^p\,\, ,\,\, 0 \le p \le m$, 
and $h_{\rm top}(f) = m \log d_1(f)$. 
\end{proposition}

Let $S^{[n]} = {\rm Hilb}^n(S)$ be the Hilbert scheme of $0$-dimensional closed subschemes of lenghts $n$ of a smooth projective surface $S$. Then $S^{[n]}$ is a smooth projective variety of dimension $2n$ by Fogarty \cite{Fo68}. Let $f \in {\rm Aut}\, (S)$. Then $f$ naturally induces an automorphism of $S^{[n]}$, which we denote by 
$f^{[n]} \in {\rm Aut}\, (S^{[n]})$. 

\begin{remark}\label{rem51} It is well known that $S^{[n]}$ is a projective hyperk\"ahler manifold of dimension $2n$ if $S$ is a projective K3 surface (\cite{Fu83} for $n=2$, \cite{Be83} for arbitrary $n$). If $S$ is an Enriques surface, then the universal cover $M$ of $S^{[n]}$, which is of covering degree $2$, is a Calabi-Yau manifold of dimension $2n$ (\cite{OS11}). If $S$ is a smooth rational surface, then $S^{[n]}$ is a smooth projective rational variety of dimension $2n$. Indeed, $S^{[n]}$ is birational to the symmetric product ${\rm Sym}^n\, (\C^2)$, the later of which is rational by classical invariant theory (see eg. \cite[Chap. 4, Theorem 2.8]{GKZ94} for details). 
\end{remark} 

\begin{proposition}\label{prop43} Let $S$ be a projective K3 surface with an automorphism $f \in {\rm Aut}\, (S)$ such that $d_1(f)$ is a Salem number. Then the automorphism $f^{[n]}$ of $S^{[n]}$ ($n \ge 2$) is primitive and of positive topological entropy.  
\end{proposition}

\begin{proof} Let $E$ be the exceptional divisor of the Hilbert-Chow morphism 
$S^{[n]} \to {\rm Sym}^n\,(S)$. By \cite{Be83}, we have an isomorphism 
$H^{2}(S^{[n]}, \Z) \simeq H^2(S, \Z) \oplus \Z(E/2)$ compatible with Hodge decomposition and the actions of $f^{[n]}$ and $f \oplus id_{\Z(E/2)}$. Thus, by Theorem \ref{thm21}, $d_1(f^{[n]}) = r_1(f^{[n]}) = r_1(f) = d_1(f)$. Hence $f^{[n]}$ is of positive topological entropy and it is primitive 
by Theorem \ref{thm4}. \end{proof}

The following consequence may be of its own interest:

\begin{theorem}\label{thm41} 

\begin{enumerate}

\item Let $M$ be a (not necessarily projective) hyperk\"ahler fourfold and $f \in {\rm Aut}\, (M)$. Then $h_{\rm top}(f) \ge 2\log \lambda_{\rm Lehmer}$
unless $h_{\rm top}(f) = 0$. Here $\lambda_{\rm Lehmer}$ is the Lehmer number (See Example \ref{ex41}).

\item There is a projective hyperk\"ahler fourfold $M$ with a primitive automorphism $f \in {\rm Aut}\, (M)$ of the smallest possible positive topological entropy $h_{\rm top}(f) = 2\log \lambda_{\rm Lehmer}$. 
\end{enumerate}
\end{theorem}

\begin{proof}

By Proposition \ref{prop42}, $d_1(f) =1$ or $d_1(f)$ is a Salem number. Guan \cite{Gu01} shows that $3 \le b_2(M) \le 8$ or $b_2(M) =23$ for hyperk\"ahler fourfolds. $d_1(f)$ is then of degree $\le 23$ again by Proposition \ref{prop42}. Thus, as remarked in Example \ref{ex41}, it follows from \cite{Mo03} that
$$d_1(f) \ge \lambda_{\rm Lehmer}\,\, ,\,\, 
{\rm unless}\,\, d_1(f) = 1\,\, .$$ 
Thus, again by Proposition \ref{prop42}, $h_{\rm top}(f) \ge 2\log \lambda_{\rm Lehmer}$ unless $h_{\rm top}(f) = 0$. This proves (1). 
 
As remarked in Example \ref{ex41}, McMullen \cite{Mc16} shows that there is a projective K3 surface $S$ with automorphism $f \in {\rm Aut}\, (S)$ such that $d_1(f)$ is the Lehmer number. By Proposition \ref{prop41}, $(S^{[2]}, f^{[2]})$ gives then a desired example in (2). \end{proof}

In the view of Proposition \ref{prop42} and Theorem \ref{thm41}, it is interesting to ask the following:
\begin{question}\label{ques41} Let $M$ be a (not necessarily projective) hyperk\"ahler manifold of dimension $2m \ge 6$ and $f \in {\rm Aut}\, (M)$. Is then 
$h_{\rm top}(f) \ge m\log \lambda_{\rm Lehmer}$ unless $h_{\rm top}(f) = 0$?
\end{question}

\section{Smooth rational manifolds with primitive birational automorphisms.}

In this section, we shall ahsow Theorems \ref{thm61} and \ref{thm63}, from which Theorem \ref{thm2} (2) and (3) follows. 

\begin{theorem}\label{thm61} Let $S$ be a projective K3 surface with $f \in {\rm Aut}\, (S)$ such that $d_1(f) > 1$. Assume that there is $\iota \in {\rm Aut}\, (S)$ such that $\iota$ is of finite order, the minimal resolution $T$ of $S/\langle \iota \rangle$ is a rational surface and $f \circ \iota = \iota \circ f$. We denote by $f_T$ the automorphism of $T$ induced by $f$. Then $T^{[n]}$ is a $2n$-dimensional smooth rational manifold and the induced automorphism $f_T^{[n]} \in {\rm Aut}\, (T^{[n]})$ is primitive and of positive entropy. 
\end{theorem}

\begin{example}\label{rem61} Let $E$ and $F$ be mutually non-isogenous elliptic curves and $S := {\rm Km}\, (E \times F)$ the Kummer K3 surface associated with the product abelian surface $E \times F$. We denote by $\omega_S$ a nowhere vanishing holomorphic $2$-form on $S$. Let $\iota \in {\rm Aut}\, (S)$ be the automorphism of $S$ of order $2$, induced by $(-1_E, 1_F) \in {\rm Aut}\, (E \times F)$. Note that $\iota^*\omega_S = -\omega_S$ and $\iota$ has fixed curves. So, the surface $T = S/\langle \iota \rangle$ is smooth and it is rational by Castelnuovo's criterion. By \cite[Remark 6.3]{Yu16}, $S$ admits an automorphism $f$ such that $d_1(f)$ is a Salem number of degree $10$. The triple $(S, f, \iota)$ then satisfies the condition in Theorem \ref{thm61}, as $\iota$ is in the center of ${\rm Aut}\, (S)$ by \cite[Proof of Lemma 1.4]{Og89}.  
\end{example}

\begin{proof} 
$T^{[n]}$ is a smooth projective rational variety of dimension $2n$ (See Remark \ref{rem51}). 
The quotient morphism $S \dasharrow T$ induces a dominant rational map $\nu : S^{[n]} \dasharrow T^{[n]}$. The map $\nu$ is equivariant under the automorphisms $f_S^{[n]}$ and $f_T^{[n]}$. As $d_1(f_S)$ is Salem number, the automorphism $f_S^{[n]}$ is primitive and of positive entropy by Proposition 
\ref{prop43}. Thus so is $f_{T}^{[n]}$ by Lemma \ref{lem32} and Corollary \ref{cor21} (1). 
\end{proof}
\begin{definition}\label{def61} Let $(\C^{\times})^{d}$ be the $d$-dimensional algebraic torus and $(t_1, t_2, \ldots, t_d)$ the standard coordinate of $(\C^{\times})^{d}$. We naturally regard as $(\C^{\times})^{d} \subset \BP^{d}$. The birational map $\varphi_{A} \in {\rm Bir}\, (\BP^{d})$ ($A = (a_{ij}) \in {\rm GL}_{d}(\Z)$) defined by 
$$\varphi_{A}^*(t_1, t_2, \ldots, t_{d}) = (\Pi_{j=1}^{d}t_j^{a_{j1}}, \Pi_{j=1}^{d}t_j^{a_{j2}}, \ldots, \Pi_{j=1}^{d} t_{j}^{a_{jd}})$$
is called the monomial map associated to $A \in {\rm GL}_{d}(\Z)$. 
\end{definition}
Note that $\varphi_{AB} = \varphi_{A} \circ \varphi_{B}$ and $\varphi_{A^{-1}} = \varphi_{A}^{-1}$ for $A, B \in {\rm GL}_{d}(\Z)$. The monomial maps are intensively studied by \cite{HP07}, \cite{FW12}, \cite{Li12} from dynamical point of views. In particular, the closed formula of the dynamical degrees of monomial maps is conjectured by \cite{HP07} and Favre and Wulcan \cite[Corollary]{FW12} and Lin \cite[Theorem 1]{Li12} confirm the conjecture: 
\begin{theorem}\label{thm62} Let $A \in {\rm GL}_{d}(\Z)$ and let $\varphi_A \in {\rm Bir}\,(\BP^d)$ be the associated monomial map. Order the eigenvalues $\alpha_j$ ($1 \le j \le d$) of $A$ as
$$|\alpha_1| \ge |\alpha_2| \ge \ldots \ge |\alpha_{d}|\,\, .$$
Then $d_k(\varphi_A) = \Pi_{j=1}^{k} |\alpha_j|$. 
\end{theorem}
\begin{theorem}\label{thm63} Let $d \ge 2$ and $\alpha >1$ be any Pisot unit of degree $d$ (Theorem \ref{thm31}). Then $\BP^{d}$ has a primitive birational automorphisms of first dynamical degree $\alpha > 1$. 
\end{theorem}
\begin{proof} Our proof is quite parallel to the proof of Theorem \ref{thm31}. 
Let $S_d(X) \in \Z[X]$ be the minimal polynomial of $\alpha$. As in the proof of Theorem \ref{thm31}, there is $M \in {\rm GL}_d(\Z)$ whose characteristic polynomial is $S_d(X)$. Order the complex zeros of $S_d(X)$, i.e., eigenvalues of $M$ as 
$$\alpha_1 := \alpha \ge |\alpha_2| \ge \ldots \ge |\alpha_d|\,\, .$$ 
Then by Theorem \ref{thm62}, the monomial map $\varphi_M$ satisfies that $d_1(\varphi_M) = \alpha$ and $d_2(\varphi_M) = \alpha |\alpha_2|$. As $\alpha$ is a Pisot unit, we have $\alpha >1$ and $|\alpha_2| < 1$. Hence,
$$d_1(\varphi_M) = \alpha >1\,\, ,\,\, d_1(\varphi_M) > d_2(\varphi_M)\,\, .$$ 
In particular, $\varphi_M$ is primitive by the second inequality and Corollary \ref{cor21} (2). This completes the proof of Theorem \ref{thm63}.  
\end{proof}
\section{Even dimensional Calabi-Yau manifolds with primitive automorphisms of positive topological entropy.}
In this section, we show the following theorem, from which Theorem \ref{thm2} (5) follows:
\begin{theorem}\label{thm71} For each positive integer $n$, there is a $2n$-dimensional Calabi-Yau manifold $M$ with a primitive automorphism $f \in {\rm Aut}\, (M)$ of positive topological entropy. 
\end{theorem}

We construct a desired Calabi-Yau manifold from an Enriques surface $W$ described in the following (see \cite{CO15}, \cite{Do16} for such Enriques surfaces and automorphisms):

\begin{proposition} \label{prop71}
There is an Enriques surface $W$ with an automorphism 
$f_W \in {\rm Aut}\, (W)$ such that $d_1(f_W) >1$. 
\end{proposition}

{\it In what follows, we choose and fix $W$ and $f_W \in {\rm Aut}\, (W)$ in Proposition \ref{prop71}.} 

Let us consider the Hilbert scheme $W^{[n]}$ and the automorphism $f_W^{[n]} \in {\rm Aut}\, (W^{[n]})$ induced by $f_W$. Let $u : M \to W^{[n]}$ be the universal cover of $W^{[n]}$, of covering degree $2$. 

Theorem \ref{thm71} will follows from the next:

\begin{proposition} \label{prop72}
\begin{enumerate}
\item $M$ is a $2n$-dimensional smooth projective Calabi-Yau manifold.
\item $M$ adimits a primitive automorphism of positive topological entropy. 
\end{enumerate}
\end{proposition}

\begin{proof} The first assertion (1) follows from \cite{OS11}, as mentioned in Remark \ref{rem51}. 

Let $\tau : S \to W$ be the universal cover of $W$. Then $S$ is a projective K3 surface and $\tau$ is an \'etale covering of degree $2$. As in Theorem \ref{thm61}, we reduce the proof of (2) to the hyperk\"ahler manifold $S^{[n]}$. 

We denote by $\iota \in {\rm Aut}\, (S)$ the covering involution of $\tau$. Let  $f_S \in {\rm Aut}\, (S)$ be any one of the two lifts of $f_W \in {\rm Aut}\, (W)$. We denote by $f_S^{[n]} \in {\rm Aut}\, (S^{[n]})$ the automorphism of $S^{[n]}$ induced by $f_S$. 

\begin{lemma} \label{lem71}
$f_S^{[n]} \in {\rm Aut}\, (S^{[n]})$ is a primitive automorphism of positive topological entropy.
\end{lemma}

\begin{proof} As $\tau$ is equivariant under $f_S$ and $f_W$, we have $d_1(f_S) = d_1(f_W)$ by Corollary \ref{cor21} (1). Thus $d_1(f_S)$ is a Salem number. Hence $f_S^{[n]}$ is primitive and of positive topological entropy by Proposition \ref{prop43}.  
\end{proof}

An essential point of proof of proposition \ref{prop72} (2) is that $S^{[n]}$, $M$ and $W^{[n]}$ are mutually related under dominant rational maps, as observed by Hayashi \cite{Ha15}. Here we briefly recall some relations from \cite[Section 2]{Ha15} with simplifications. 

Let $V^n$ (resp. $V^{(n)}$) be the product (resp. the symmetric product) of a smooth projective manifold $V$ and $\Sigma_n$ the symmetric group of degree $n$. Then, by definition, the natural morphism $\mu_V : V^n \to V^{(n)}$ is a finite Galois cover of Galois group $\Sigma_n$ and we have also the Hilbert-Chow morphism $\nu_V : V^{[n]} := {\rm Hilb}^n\, (V) \to V^{(n)}$. 

Then, in our situation, we have a natural finite morphism $q : S^n \to W^{(n)}$, which is Galois whose Galois group is
$$G := \{\iota_I\, |\, I \subset \{1, 2, \ldots, n\}\} \cdot \Sigma_n = \pi_1(W)^{n} \cdot \Sigma_n \subset {\rm Aut}\, (S^n)\,\, .$$
Here $\Sigma_n$ is the group of permutations of the $n$ factors of $S^n$ and $\iota_I \in {\rm Aut}\, (S^n)$ is defined by 
$$\iota_I((x_i)_{i=1}^{n}) = (y_i)_{i=1}^{n}\,\, ; \,\, y_i := \iota(x_i)\,\, (i \in I)\,\, ,\,\, y_j := x_j\,\, (j \not\in I)\,\, .$$ 
$f_{S^n} := (f_S, f_S, \ldots, f_S) \in {\rm Aut}\, (S^n)$ commutes with each element of $G$, as $f_S \circ \iota = \iota \circ f_S$. The group $G$ is of order $2^n \cdot n!$ with a normal subgroup $H$ of index two:
$$H := \{\iota_I\, |\, |I| \in 2\Z\} \cdot \Sigma_n \subset G\,\, .$$
Then we have finite morphisms
$$\overline{v} : S^{(n)} = S^n/\Sigma_n \to \overline{M} := S^n/H\,\, ,\,\, \overline{u} : \overline{M} \to W^{(n)} = S^n/G\,\, .$$
The first morphism is of degree $2^{n-1}$, which is not Galois if $n \ge 3$, but the second one is Galois with Galois group $G/H = \langle [\iota_{\{1\}}] \rangle \simeq \Z/2$. As noticed in \cite[Section 1]{OS11}, we have  
$$\pi_1(W^{[n]}) \simeq \pi_1(W^{(n)}) \simeq H_1(W, \Z) = \pi_1(W) = \langle \iota \rangle\,\, .$$
Here the first isomorphism is given by the Hilbert-Chow morphism and the second one is given by $[\iota_{\{1\}}] \leftrightarrow \iota$. 
Thus, $\overline{u}$ is the universal cover of $W^{(n)}$. The universal cover of $W^{[n]}$ is then given by 
$$u = \overline{u} \times_{W^{[n]}} \nu_{W} : M = \overline{M} \times_{W^{(n)}} W^{[n]} \to W^{[n]}\,\, .$$ 

The following proposition completes the proof of Proposition \ref{prop72} (2): 

\begin{proposition} \label{lem72}
At least one of the two lifts, say $f_M \in {\rm Aut}\, (M)$, of $f_{W^{[n]}}$ is primitive and of positive topological entropy. 
\end{proposition}

\begin{proof} Note that $u$ is equivariant with respect to $f_M$ and $f_{W^{[n]}}$ for any lift $f_M$ of $f_{W^{[n]}}$. 

Let $f_{W^{(n)}}$ (resp. $f_{S^{(n)}}$) be the automorphism of $W^{(n)}$ (resp. $S^{(n)}$) induced by $f_W \in {\rm Aut}\, (W)$ (resp. $f_S \in {\rm Aut}\, (S)$, any one of the two lifts of $f_W$). 
Then, by the description above, $f_{S^{(n)}}$ descends to $f_{\overline{M}} \in {\rm Aut}\, (\overline{M})$ so that $\overline{v}$ and $\overline{u}$ are equiariant with respect to $f_{S^{(n)}}$, $f_{\overline{M}}$ and $f_{W^{(n)}}$. 
In particular, 
$$\tilde{f}_M := f_{\overline{M}} \times_{f_{W^{(n)}}} f_{W^{[n]}} \in {\rm Aut}\, (M)\,\, ,$$ 
under the description $M = \overline{M} \times_{W^{(n)}} W^{[n]}$ of the universal covering space above, and $\tilde{f}_M$ is one of the two lifts of $f_{W^{[n]}}$ to $M$. We choose this $\tilde{f}_M$ as $f_M$. 
Then the natural rational map 
$$q : S^{[n]} \to \overline{M} \dasharrow M$$
is equivariant with respect to $f_{S^{[n]}}$ and $f_M$. As $q$ is a generiaclly finite dominant rational map and $f_{S^{[n]}}$ is primitive and of positive entropy, so is $f_M$ by Corollary \ref{cor21} (1) and Lemma \ref{lem32}. 
\end{proof}

This completes the proof of Proposition \ref{prop72}. \end{proof}

\begin{remark} \label{rem71}
The universal covering Calabi-Yau $2n$-folds of the Hilbert scheme $W^{[n]}$ of an Enriques surface $W$ may have other interesting geometric structures. See \cite{Ha15} for some interesting observations. 
\end{remark}

\section{Odd dimensional Calabi-Yau manifolds with primitive birational automorphisms of first dynamical degree $>1$.}

{\it Throughout this section we fix an $n$-dimensional Calabi-Yau manifold of Wehler type $M$}, i.e., a general hypersurface $M$ of multi-degree $(2, 2, \ldots, 2)$ in 
$$(\BP^1)^{n+1} = \BP_1^1 \times \ldots \times \BP_k^1 \times \ldots \BP_{n+1}^1$$ with $n \ge 3$. $M$ is indeed an $n$-dimensional Calabi-Yau manifold. 

Our main result of this section is the following:
\begin{theorem} \label{thm81}
Any odd dimensional Calabi-Yau manifold $M$ of Wehler type has a primitive birational automorphism of first dynamical degree $> 1$. 
\end{theorem}
Theorem \ref{thm2} (4) follows from Theorems \ref{thm81} and \ref{thm71}. Our proof of Theorem \ref{thm81} is much inspired by arguments in \cite{Bi16}, and is based on Theorem \ref{thm82} below (the main theorem of \cite{CO15}).

For the statement of Theorem \ref{thm82}, we recall that 
$$N^1(M) \simeq {\rm Pic}\, M = \oplus_{k=1}^{n+1} \Z h_k \simeq \Z^{n+1}
\,\, , \,\, \overline{{\rm Amp}}\,(M) = \oplus_{k=1}^{n+1} \R_{\ge 0} h_k\,\, .$$
Here $h_k$ is the pullback of the hyperplane class of $\BP_k^1$ under the natural projection $M \to \BP_k^1$. As $f : M \dasharrow M$ is isomorphic in codimension one for $f \in {\rm Bir}\, (M)$ (see eg. \cite{Ka08}, which is also very crucial in the proof of Theorem \ref{thm82}), we have a natural contravariant representation 
$$\rho : {\rm Bir}\, (M) \to {\rm GL}\, (N^1(M)) \simeq {\rm GL}\, (\Z^{n+1})\,\, ; \,\, f \mapsto f^*\,\, .$$
Let $p_k : M \to P_k$ ($1 \le k \le n+1$) be the natural projection, where $P_k$ is the product manifold $(\BP^1)^{n}$ obtained by removing the $k$-th factor of $(\BP^1)^{n+1}$. The morphism $p_k$ is of degree $2$ and the covering involution $\iota_k$ of $p_k$ is a (not biregular) birational automorphism of $M$ of order two. 
\begin{theorem} \label{thm82} Let $M$ be a Calabi-Yau manifold of Wehler type of dimension $\ge 3$. Then:

\begin{enumerate}

\item ${\rm Aut}\, (M) = \{id\}$, ${\rm Bir}\, (M) = \langle \iota_1, \ldots, \iota_{n+1} \rangle$ and $({\rm Bir}\, (M), \{\iota_k\}_{k=1}^{n+1})$ forms a Coxeter group isomorphic to the universal Coxeter group $UC(n+1)$. 

\item The representation $\rho$ is faithful and coincides with the dual of the geometric representation of the universal Coxter group $({\rm Bir}\, (M), \{\iota_k\}_{k=1}^{n+1}) \simeq UC(n+1)$. In particular, there is an integral non-degenerate bilinear form $(*, **)$ on $N^1(M) \simeq \Z^{n+1}$ of signature $(1, n)$, with respect to which the action of ${\rm Bir}\, (M)$ is orthogonal and $\iota_k$ is an othogonal reflection with respect to the hypersurface $\oplus_{j \not= k} \R h_j$ (see \cite[Section 2.2.2]{CO15} for the explicit description of $(*, **)$). 

\item For any effective movable divisor $D$, there are a nef divisor $H$ and $g \in {\rm Bir}\, (M)$ such that $g^*D = H$ in $N^1(M)$. In particular, $H$ is semi-ample and the image of $\Phi_{|mH|}$ for large $m >0$ is a smooth projective variety, which is either $(\BP^1)^{k}$ ($0 \le k \le n-1$) or $M$ (see the description of the nef cone $\overline{{\rm Amp}}\,(M)$ above). 
\end{enumerate}
\end{theorem}

The following proposition is crucial for Theorem \ref{thm81}:
\begin{proposition} \label{prop81}
Assume that $M$ is a Calabi-Yau manifold of Wehler type of odd dimension $n = 2k+1 \ge 3$. Then, there is $f \in {\rm Bir}\, (M)$ such that the characteristic polynomial of $f^*|N^1(M)$ is irreducible over $\Z$ and it is a Salem polynomial of degree $n+1 = {\rm rank}\, N^1\, (M)$. In particular, there is no non-trivial $f$-stable $\Q$-linear subspace of 
$N^1(M)_{\Q} := N^1(M) \otimes \Q$. 
\end{proposition}
\begin{proof}
By Theorem \ref{thm82} (2), we may and will regard ${\rm Bir}\, (M) \subset {\rm O}(N^1(M)_{\R})$ with respect to the bilinear form $(*, **)$.
By ${\rm Bir}\, (M) = \langle \iota_k \rangle_{k=1}^{n+1}$ and the description of $\iota_k^*$ on $N^1(M)$ in Theorem \ref{thm82} (2), the action of ${\rm Bir}\, (M)$ on $N^1(M)_{\R}$ is irreducible. Thus the result follows from Theorem \ref{prop41} (2). 
\end{proof}

Theorem \ref{thm81} will follow from the next:

\begin{proposition} \label{prop82}
Let $M$ and $f \in {\rm Bir}\, (M)$ be as in Proposition \ref{prop81}. Then 
$f$ is primitive and $d_1(f) > 1$. 
\end{proposition}
\begin{proof}
As $M$ is a smooth Calabi-Yau manifold, $f$ is isomorphic in codimension one. 
Thus $d_1(f) = r_1(f) > 1$ by Remark \ref{rem20} (2) and by the choice of $f$. 

We show that $f$ is primitive. First, we observe the following:

\begin{lemma} \label{lem80}
There are no smooth projective variety $B$ with a dominant rational map $\pi : M \dasharrow B$ of connected fibers and $f_B \in {\rm Bir}\, (B)$ such that $\pi \circ f = f_B \circ \pi$, $0 < \dim\, B < \dim\, M$ and resolutions of general fibers of $\pi$ are of Kodaira dimension $0$. 
\end{lemma}

\begin{proof} Assuming to the contrary that there are such $\pi : M \dasharrow B$ and $f_B \in {\rm Bir}\, (B)$, we derive a contradiction. 

Let $\nu : \tilde{M} \to M$ be a resolution of the indeterminacy $I(\pi)$ of $\pi$ and $\tilde{\pi} : \tilde{M} \to B$ the induced morphism. Let $H$ be a very ample divisor on $B$. Then, by definition, $\pi^*H = \nu_*(\tilde{\pi}^*H)$ is a movable divisor on $M$. By Theorem \ref{thm82} (3), there is $g \in {\rm Bir}\, (M)$ such that $L := g^*(\pi^*H)$ is nef. Note that $|L|$ is free by the first part of Theorem \ref{thm82} (3). So, by replacing $f$ by $g \circ f \circ g^{-1}$ and $\pi$ by $\pi \circ g^{-1}$, we may and will assume from the first that $\pi : M \dasharrow B$ is given by a linear subsystem $\Lambda$ of a free linear system $|L|$, i.e., $\pi = \Phi_{\Lambda}$. In this replacement, we also note that $g \circ f \circ g^{-1}$ has the same characteristic polynomial as $f$ on $N^1(M)$. 

As $K_M = 0$ and $\kappa(\tilde{F}) = 0$ by our assumption, we have $\dim\, B = \kappa(M, L)$ by \cite[Theorem 5.11]{Ue75} (see also \cite[Th\'eor\`eme 2.3]{AC13}). As $\Lambda \subset |L|$, the morphism $\Phi_{|L|} : M \to B'$ given by the complete free linear system $|L|$ factors through $\pi$ and we have a dominant rational map $q : B' \dasharrow B$ such that $\pi = q \circ \Phi_{|L|}$. We have then $\dim\, B' = \dim\, B$ by $\kappa(M, L) = \dim\, B$. Hence $q$ is birational, as $\pi$ has connected fibers.  

So, by replacing $\pi$ by $\Phi_{|mL|}$ for large $m >0$, we may and will assume that 
$$B = {\rm Proj}\,\oplus_{k\ge 0} H^0(M, \sO_M(kL))$$ 
and $\pi : M \to B$ is a morphism. As $f$ is isomorphic in codimension one, the pullback $f^*$ then naturally induces an isomorphism 
$$f_* : B = {\rm Proj}\,\oplus_{k\ge 0} H^0(M, \sO_M(kL)) \simeq {\rm Proj}\, \oplus_{k\ge 0} H^0(M, \sO_M(kf^*L)) = B\,\, ,$$
i.e., $f_B := f_* \in {\rm Aut}\, (B)$. This $f_B$ satisfies $\pi \circ f = f_B \circ \pi$. As $f$ is isomorphic in codimension one and $f_B$ and $\pi$ are {\it morphisms}, it follows that $0 \not= \pi^*N^1(B)_{\Q} \subset N^1(M)_{\Q}$ is $f$-stable. Hence $\pi^*N^1(B)_{\Q} = N^1(M)_{\Q}$, by the irreduciblity of the action of $f$. However, this is impossible, as $\pi^*N^1(B)_{\Q}|F = 0$ but $N^1(M)_{\Q}|F \not= 0$ for general fibers $F$ of the morphism $\pi$. 
\end{proof}

From now, we closely follow an argument in \cite{Bi16}. 

\begin{lemma} \label{lem81}
Let $P \in M$ be a general point. Then, the Zariski closure of the orbit $\{f^n(P)\, |\, n \in \Z\}$ is $M$. In particular, there is no dominant rational map $\pi : M \dasharrow B_1$ of connected fibers and $f_{B_1} \in {\rm Bir}\, (B_1)$ such that $\pi_1 \circ f = f_{B_1} \circ \pi_1$, $0 < \dim\, B_1 < \dim\, M$ and resolutions of general fibers of $\pi_1$ are of general type.
\end{lemma}

\begin{proof} If otherwise, there are smooth projective variety $C$ and a domonant rational map $\tau : M \dasharrow C$ such that $\tau \circ f = \tau$ by \cite[Th\'eor\`eme 4.1]{AC13}. As $f$ is isomorphic in codimension one, 
it follows that $f^*v = v$ for all $v \in \tau^*N^1(C)$ by \cite[Lemma 4.5]{Bi16}, a contradiction to the irreducibility of $f^*$ on $N^1(M)$. 

We show the last statement. If we would have a rational map $\pi_1 : M \dasharrow B_1$ as in Lemma \ref{lem81}, then 
$\{f_{B_1}^n(Q)\, |\, n \in \Z\}$ would be Zariski dense in $B_1$ for generic $Q \in B_1$, as so is for $f$ and $\pi_1$ is dominant. The last statement now follows from \cite[Proposition]{Bi16}.
\end{proof}

\begin{lemma} \label{lem82}
Assume that there are a smooth projective variety $B_2$, a dominant rational map $\pi_2 : M \dasharrow B_2$ of connected fibers and $f_{B_2} \in {\rm Bir}\, (B_2)$ such that $\pi_2 \circ f = f_{B_2} \circ \pi_2$ and 
$0 < \dim\, B_2 < \dim\, M$. 
Then, there are a smooth projective variety $B_3$, a dominant rational map $\pi_3 : M \dasharrow B_3$ of connected fibers and $f_{B_3} \in {\rm Bir}\, (B_3)$ such that $\pi_3 \circ f = f_{B_3} \circ \pi_3$, $0 < \dim\, B_3 < \dim\, M$ and resolutions of general fibers of $\pi_3$ are of Kodaira dimension $0$. 
\end{lemma}

\begin{proof} Resolutions of general fibers of $\pi_2$ are of non-negative Kodaira dimension, by the adjunction formula applied for a resolution of the indeterminacy of $\pi_2$. So, one can consider the relative Kodaira fibration $\pi_3 : M \dasharrow B_3$ over $B_2$ and $\pi_3$ satisfies all requirement except perhaps $0 < \dim\, B_3 < \dim\, M$. We have also $0 < \dim\, B_3 < \dim\, M$, as resolutions of general fibers of $\pi_2$ are not of general type by Lemma \ref{lem81}. 
\end{proof}

By Lemmas \ref{lem80} and \ref{lem82}, $f$ has to be primitive. \end{proof}

\end{document}